\documentclass{birkjour}
\usepackage{amssymb,amsmath,setspace,amsthm}

 \newtheorem{thm}{Theorem}[section]
 \newtheorem{cor}[thm]{Corollary}
 \newtheorem{lemma}[thm]{Lemma}
 \newtheorem{prop}[thm]{Proposition}
 \theoremstyle{definition}
 \newtheorem{mydef}[thm]{Definition}
 \theoremstyle{remark}
 \newtheorem{remark}[thm]{Remark}
\newcommand{\twomat}[4]{\left(\begin{array}{cc} #1 & #2 \\  #3 & #4 \end{array}\right)}

\title[Perturbed periodic Jacobi matrices]{Eigenvalues for perturbed periodic Jacobi matrices by the Wigner-von Neumann approach}

\author{Edmund Judge}
\address{SMSAS, Cornwallis Building, University of Kent, Canterbury,\\ Kent CT2 7NF, United Kingdom}
\email{ej75@kent.ac.uk}
\author{Sergey Naboko}
\address{ Depart. of Math. Physics, Institute of Physics, St.Petersburg University, \\ St.Petergoff, St.Petersburg, 198904, Russia} \email{sergey.naboko@gmail.com}
\author{Ian Wood}
\address{SMSAS, Cornwallis Building, University of Kent, Canterbury,\\ Kent CT2 7NF, United Kingdom}
\email{i.wood@kent.ac.uk}
\date{}

\begin{document}

\begin{abstract}
The Wigner-von Neumann method, which has previously been used for perturbing continuous Schr\"{o}dinger operators, is here applied to their
discrete counterparts. In particular, we consider perturbations of arbitrary $T$-periodic Jacobi matrices. The asymptotic behaviour of the
subordinate solutions is investigated, as too are their initial components, together giving a general technique for embedding eigenvalues,
$\lambda$, into the operator's absolutely continuous spectrum. Introducing a new rational function, $C(\lambda;T)$, related to the periodic
Jacobi matrices, we describe the elements of the a.c. spectrum for which this construction does not work (zeros of $C(\lambda;T)$); in
particular showing that there are only finitely many of them.
\end{abstract}

\thanks{Edmund Judge was supported by the Engineering and Physical Sciences Research Council (grant EP/M506540/1).
Sergey Naboko was supported by the Russian Science Foundation (grant  15-11-30007),
NCN 2013/09/BST1/04319 and Marie Curie grant (2013-2014 yy) PIIF-GA-2011-299919.}

\maketitle
\section{Introduction}

First published in 1929, the Wigner-von Neumann method provides a way of embedding eigenvalues into the absolutely continuous spectrum of a
one-dimensional Schr\"{o}dinger operator~{\cite{3}}. Specifically, the operator is perturbed by a potential of the form $$\frac{c\sin(2\omega
x+\varphi)}{x}$$ causing the eigenvalue $E=\omega^2$ to become embedded in the interval $[0,\infty)$ of a.c. spectrum. Since then the method
has been adapted to embed multiple eigenvalues $E_i=\omega_i^2$, $i=1,2,\dots,N$ \cite{4} using a single potential
$$\sum\limits_{i=1}^N\frac{c_i\sin{(2\omega_ix+\varphi_i)}}{x}. $$ More recently, the technique has been employed on periodic
Schr\"{o}dinger operators which have several or infinitely many bands of absolutely continuous spectrum~\cite{1,9a, 9b}. In particular, the
unperturbed operators have the form $$-\frac{d^2}{dx^2}+Q_{per}(x)$$ where $Q_{per}(x+T)=Q_{per}(x)$ for some period $T$. Note that another method based on the explicit solution of the inverse problem \cite{9} also allows to embed multiple (but finitely many) eigenvalues into the essential spectrum with a potential similar to the above.

The purpose of this paper is to extend the aforementioned ideas to the discrete analogue of Jacobi operators, and produce a new technique for embedding a single eigenvalue into one of the bands of the periodic operator's essential spectrum. Note that there already exist several other approaches for embedding eigenvalues into the essential spectrum of both Schr\"{o}dinger operators~\cite{9d,9h,Sim97} and Jacobi
matrices~\cite{9e,9k,9f,9g,9j,9i}. However, the big advantage of the Wigner-von Neumann method is that it gives an explicit, and relatively simple, formula for the potential and eigenvector, even for the periodic case.

A Jacobi operator is defined to be a tri-diagonal infinite matrix  which is considered as an operator on $l^2(\mathbb{N};\mathbb{C})$. We
consider only real bounded Hermitian Jacobi operators and assume without loss of generality that the off-diagonal entries are positive.
Moreover, we assume the matrix to be $T$-periodic. Then, our operator, $J_T$, has the form \begin{equation}\label{4.10}
\left(\begin{array}{ccccccccccccc}
 b_1&a_1&\\
 a_1&b_2&a_2&\\
 &a_2&b_3&a_3&\\
 &&\ddots&\ddots&\ddots&\\
 &&&a_{T-1}&b_T&a_T\\
 &&&&a_T&b_1&a_1&\\
 &&&&&a_1&b_2&a_2&\\
 &&&&&&\ddots&\ddots&\ddots&\\
 &&&&&&&a_{T-1}&b_{T}&a_T\\
 &&&&&&&&a_T&b_1&a_1&\\
 &&&&&&&&&a_1&b_2&a_{2}&\\
 &&&&&&&&&&\ddots&\ddots&\ddots
 \end{array}\right),\end{equation} with $a_i\in\mathbb{R}, a_i>0$ for all $i$. From Section~\ref{sec3} onwards we will assume for simplicity
 that $b_i=0$ for all $i$. The method also works for more general cases with complex entries and non-zero diagonal, but we stick to the
 simple case to make our constructions more transparent.

Our goal is to make an ansatz for a possible eigenvector (introduced in Section 5) and an ansatz for the potential (Section 6). We then
establish the asymptotics of the potential needed to realize our ansatz. Additionally, we must confirm that the subordinate solution
constructed in this way also satisfies the initial equations encoded within the Jacobi matrix (Section 7), thus giving an embedded
eigenvalue. Section 2 contains some general results on $T$-periodic Jacobi operators, while Sections 3 contains some preliminary results for
the construction. In Section 4 we introduce the aforementioned function $C(\lambda;T)$ and analyze its properties, in particular that it is a
rational function of $\lambda$.

\section{Preliminary results on the spectrum of a period-T Jacobi operator}

Before we consider perturbations, it is best to state some results describing the structure of the spectrum of unperturbed $T$-periodic
Jacobi operators.

It is not hard to see that the spectrum of a period-$T$ Jacobi matrix, $J_T$, is such that \begin{multline*}\sigma(J_T)\subseteq
[-\max\{a_1+a_2,a_2+a_3,\dots,a_T+a_1\}-\min\{b_1,\dots,b_T\},\\\max\{a_1+a_2,a_2+a_3,\dots,a_T+a_1\}+\max\{b_1,\dots,b_T\}].\end{multline*}
The inclusion is sharp for period-$1$ and period-$2$ Jacobi operators, which can easily be proved. Moreover, the operator is self-adjoint and
in the case of the vanishing diagonal its spectrum is symmetric w.r.t. zero.

 The next elementary lemma gives information which will be useful in determining the entries of the monodromy matrix associated to $J_T$.
\begin{lemma}
 Let $$A_s=\left(\begin{array}{cc}
 0&1\\
 c_s&\frac{\lambda-b_s}{a_s}\end{array}\right)$$ with $a_s,c_s\neq 0$. Let $m\in\mathbb{N}$ and $$A(\lambda)=\left(\begin{array}{cc}
 a_{11}(\lambda)&a_{12}(\lambda)\\
 a_{21}(\lambda)&a_{22}(\lambda)\end{array}\right)=\prod\limits_{s=1}^m A_s.$$ Then, for $m\geq 2$,
 \begin{align*}
 a_{11}(\lambda)&=c_1\frac{\lambda^{m-2}}{\prod\limits_{s=2}^{m-1} a_s}+P_{m-3}(\lambda),&
 a_{12}(\lambda)&=\frac{\lambda^{m-1}}{\prod\limits_{s=1}^{m-1} a_s}+P_{m-2}(\lambda),\\
 a_{21}(\lambda)&=c_1\frac{\lambda^{m-1}}{\prod\limits_{s=2}^{m} a_s}+\widetilde{P}_{m-2}(\lambda),&
 ~{and}~~~~~a_{22}(\lambda)&=\frac{\lambda^{m}}{\prod\limits_{s=1}^{m} a_s}+P_{m-1}(\lambda),
 \end{align*}
 where $P_{m-1}(\lambda),P_{m-2}(\lambda),\widetilde{P}_{m-2}(\lambda)$ and $P_{m-3}(\lambda)$ are real polynomials in $\lambda$ of degree
 less than or equal to $m-1,m-2, m-2$ and $m-3$, respectively, and $P_k(\lambda)=0$ for $k<0$.

 \end{lemma}

 \begin{proof}
We use induction on $m$, $m\geq2$. For the base case of $m=2$ we have
$a_{11}=c_1,a_{12}=\frac{\lambda-b_1}{a_1},a_{21}=c_1\frac{\lambda-b_2}{a_2}$ and $a_{22}=c_2+\frac{(\lambda-b_2)(\lambda-b_1)}{a_2a_1}.$
Immediately, they satisfy the hypothesis.
  Now assume the result holds up to $m=k$. To prove the result for $m=k+1$, observe that by induction,

  {\setstretch{2}
\begin{alignat*}{1}
\prod\limits_{s=1}^{k+1} \left(\begin{array}{cc}
0&1\\
c_s&\frac{\lambda-b_s}{a_s}
\end{array}\right)&=\left(\begin{array}{cc}
0&1\\
c_{k+1}&\frac{\lambda-b_{k+1}}{a_{k+1}}
\end{array}\right)\prod\limits_{s=1}^{k} \left(\begin{array}{cc}
0&1\\
c_s&\frac{\lambda}{a_s}
\end{array}\right)\\
&\hspace{-40pt}=\left(\begin{array}{cc}
0&1\\
c_{k+1}&\frac{\lambda-b_{k+1}}{a_{k+1}}
\end{array}\right)\left(\begin{array}{cc}
c_1\frac{\lambda^{k-2}}{\prod_{s=2}^{k-1} a_s}+P_{k-3}&\frac{\lambda^{k-1}}{\prod_{s=1}^{k-1} a_s}+P_{k-2}\\
c_1\frac{\lambda^{k-1}}{\prod_{s=2}^{k} a_s}+\widetilde{P}_{k-2}&\frac{\lambda^{k}}{\prod_{s=1}^{k} a_s}+P_{k-1}
\end{array}\right)\\
&\hspace{-40pt}=\left(\begin{array}{cc}
c_1\frac{\lambda^{k-1}}{\prod_{s=2}^{k} a_s}+\widetilde{P}_{k-2}&\frac{\lambda^{k}}{\prod_{s=1}^{k} a_s}+P_{k-1}\\
c_1\frac{\lambda^{k}}{\prod_{s=2}^{k+1}a_s}+\widetilde{P}_{k-1}&\frac{\lambda^{k+1}}{\prod_{s=1}^{k+1} a_s}+P_{k}
\end{array}\right).\qedhere\end{alignat*}}
 \end{proof}

\begin{cor}\label{per71}
  Let $M$ be the monodromy matrix for an arbitrary period-$T$ operator, i.e. $$M(\lambda)=\left(\begin{array}{cc}
 m_{11}(\lambda)&m_{12}(\lambda)\\
 m_{21}(\lambda)&m_{22}(\lambda)\end{array}\right):=B_{T}(\lambda)B_{T-1}(\lambda)\dots B_{1}(\lambda)$$ and $B_i(\lambda)$ are the transfer
 matrices given by $$B_{i}(\lambda):=\left(\begin{array}{cc}
0&1\\
\frac{-a_{i-1}}{a_i}&\frac{\lambda-b_i}{a_i}\end{array}\right), \lambda\in\mathbb{C},$$ where $i=1,2,\dots,T,$ with $a_0:=a_T$. Then,
$det(M(\lambda))=1$ and for all $T\geq 1$ we have \begin{align*}
m_{11}(\lambda)&=-a_T\frac{\lambda^{T-2}}{\prod\limits_{s=1}^{T-1} a_s}+P_{T-3}(\lambda),&
m_{12}(\lambda)&=\frac{\lambda^{T-1}}{\prod\limits_{s=1}^{T-1}a_s}+P_{T-2}(\lambda),\\
m_{21}(\lambda)&=-\frac{\lambda^{T-1}}{\prod\limits_{s=1}^{T-1}
a_s}+\widetilde{P}_{T-2}(\lambda),&m_{22}(\lambda)&=\frac{\lambda^T}{\prod\limits_{s=1}^T a_s}+P_{T-1}(\lambda),
 \end{align*}
 where $P_{T-1}(\lambda),P_{T-2}(\lambda), \widetilde{P}_{T-2}(\lambda)$ and $P_{T-3}(\lambda)$ are real polynomials in $\lambda$ of degree
 less than or equal to $T-1,T-2, T-2$ and $T-3$, respectively.
 \end{cor}

It is a classical fact (see, for example, \cite{7} for Schr\"{o}dinger operators and \cite{6} for Jacobi matrices) that for
$\lambda\in\mathbb{R}$, $|Tr\left(M(\lambda)\right)|<2$ implies that $\lambda$ lies in the absolutely continuous spectrum,
$\sigma_{a.c.}(J_T)$, while if $|Tr\left(M(\lambda)\right)|>2$, then $\lambda$ lies in the resolvent or the point spectrum:
$\rho(J_T)\cup\sigma_p(J_T)$. Furthermore, we can canonically partition the points in the complex plane into three categories: hyperbolic,
elliptic, parabolic.

\begin{mydef}\label{genint}
The hyperbolic points are those $\lambda\in\mathbb{C}$ that produce a monodromy matrix with two eigenvalues, $\mu_1,\mu_2$ such that
$|\mu_1|>1$ and $|\mu_2|<1$; elliptic points those that produce two distinct eigenvalues of modulus one; and parabolic points those that
produce one eigenvalue of algebraic multiplicity two, i.e. $Tr(M(\lambda))=\pm2$. Moreover, we define the generalised interior of the essential
spectrum, denoted $\sigma_{ell}$, to be the set of elliptic points.
\end{mydef}
\begin{remark}
For $\lambda\in\mathbb{R}$ we can distinguish the hyperbolic, elliptic and parabolic cases by $|Tr(M(\lambda))|>2, |Tr(M(\lambda))|< 2,
|Tr(M(\lambda))|=2$, respectively.
\end{remark}

\begin{lemma}\label{4.20}
All points in $\mathbb{C}^+\cup\mathbb{C}^-$ for an arbitrary $T$-periodic Jacobi operator belong to the hyperbolic region. In particular,
all parabolic (and elliptic) points are real.
\end{lemma}

\begin{proof}
Let $\lambda\in\mathbb{C}\setminus\mathbb{R}$ and consider the Weyl vector, ${{f}}_\lambda$, defined as
$${{f}}_\lambda:=(J_T-\lambda)^{-1}{{e}}_1$$ where ${{e}}_1=(1,0,0,\dots)$. Obviously, this Weyl vector belongs to $l^2$ and therefore its
components are decaying. If $\lambda$ belongs to either the elliptic or parabolic regions then trivial analysis of the powers of the
monodromy matrix leads to the fact that no solution of the recurrence relation $$a_{n-1}u_{n-1}+b_nu_n+a_nu_{n+1}=\lambda u_n, n\geq 2$$
decays. Therefore all non-real $\lambda$ are hyperbolic points.\qedhere

\end{proof}
The following result will not surprise specialists in the area, but to the best of our knowledge there is no proof in the literature. Of
course, it is a folklore-type result.

 \begin{lemma}
We consider a family of period-$T$ Jacobi operators
\begin{equation}J_{\epsilon,\eta}:=\left(\begin{array}{ccccccccccc}
 b_1+\epsilon&a_1+\eta&\\
 a_1+\eta&b_2&a_2\\
 &\ddots&\ddots&\ddots\\
 &&a_{T-1}&b_T&a_T\\
 &&&a_{T}&b_1+\epsilon&a_1+\eta\\
 &&&&a_1+\eta&b_2&a_2\\
 &&&&&\ddots&\ddots&\ddots\\
 \end{array}\right) \end{equation}
depending on the two parameters $\epsilon$ and $\eta$.
Then there exists an open dense set $D$ in $\mathbb{R}^2$ such that for all $(\epsilon,\eta)\in D$ the essential spectrum of $J_{\epsilon,\eta}$
consists of $T$ distinct real intervals.
\end{lemma}

\begin{proof} We will refer to the situation that $\sigma_{ess}(J_{\epsilon,\eta})$
consists of $T$ distinct real intervals as the non-degenerate case. The proof will consist of two parts: We  show that non-degeneracy is stable under small perturbations, while on the other hand the degenerate case is not stable. We initially introduce some notation.

Define the transfer matrices for $J_{\epsilon,\eta}$ as $$B_{i}(\lambda):=\left(\begin{array}{cc}
0&1\\
\frac{-a_{i-1}}{a_i}&\frac{\lambda-b_i}{a_i}\end{array}\right),$$
for $i=3,4,\dots,T,$ and
$$B_{1,\epsilon,\eta}(\lambda):=\left(\begin{array}{cc}
0&1\\
\frac{-a_{T}}{a_1+\eta}&\frac{\lambda-b_1-\epsilon}{a_1+\eta}\end{array}\right), \quad B_{2,\epsilon,\eta}(\lambda):=\left(\begin{array}{cc}
0&1\\
-\frac{a_{1}+\eta}{a_2}&\frac{\lambda-b_2}{a_2}\end{array}\right).$$
Let $M_0$ be the monodromy matrix for $J_{0,0}$, i.e.~$M_0:=B_{T}B_{T-1}\dots B_{2,0,0} B_{1,0,0}$, and $M_{\epsilon,\eta}$ the monodromy matrix for $J_{\epsilon,\eta}$, i.e.~
\begin{equation}\label{Meps}M_{\epsilon,\eta}:=B_{T,0,0}B_{T-1,0,0}\dots B_{3,0,0} B_{2,\epsilon,\eta}B_{1,\epsilon,\eta}= M_0 B_{1,0,0}^{-1}B_{2,0,0}^{-1}B_{2,\epsilon,\eta}B_{1,\epsilon,\eta}.\end{equation}

 By Corollary~\ref{per71} we have $$M_0(\lambda)=\left(\begin{array}{cc}
 p_1(\lambda)&p_2(\lambda)\\
 p_3(\lambda)&p_4(\lambda)\end{array}\right),$$
where $p_1(\lambda),p_2(\lambda),p_3(\lambda)$ and $p_4(\lambda)$ are real polynomials in $\lambda$ of order $T-2$,
 $T-1,T-1$ and $T$ respectively. Then, as $\lambda\mapsto \left(Tr(M_0(\lambda))\pm 2\right)$ are two polynomials each of degree $T$ in
 $\lambda$, there are at most $2T$ real zeros of these functions, providing at most $T$ intervals of a.c. spectrum.  Recall from Lemma~\ref{4.20} that all of the parabolic points
for $J_T$ are real.

 {\underline{(Step One)}} It needs to be shown that if $J_{0,0}$ is non-degenerate, then adding sufficiently small $\epsilon, \eta$ to the operator does not cause two
 previously distinct parabolic points to overlap. The argument is simple: For each pair of distinct parabolic points $(\lambda_j,\lambda_k)$ of $J_{0,0}$ there
 exists $\delta_{j,k}>0$ such that for the corresponding parabolic points of $J_{\epsilon,\eta}$ we have $\lambda_j(\epsilon,\eta)\neq\lambda_k(\epsilon,\eta)$ for $|\epsilon|, |\eta|<\delta_{j,k}$, using  the fact that the roots of a polynomial depend continuously on its coefficients. (See, for example, Appendix A in \cite{2}.)
Then, since there are at most $2T$ parabolic points in total, we can define $$\delta:=\min_{j,k}\delta_{j,k}>0$$ which implies
$$\lambda_{m}(\epsilon,\eta)\neq \lambda_n(\epsilon,\eta)$$ for all $|\epsilon|, |\eta| <\delta, m,n\in \{1,\dots,2T\}$, $m\neq n$. This shows that the non-degenerate case is stable.

{\underline{(Step Two)}}   We now show that the case where two of the intervals of essential spectrum of $J_{0,0}$ overlap is unstable. Let $\lambda_0$ be a parabolic point for $J_{0,0}$. We will only consider the case when ${\rm{Tr}}(M_0(\lambda_0))=2$, the case ${\rm{Tr}}(M_0(\lambda_0))=-2$ can be dealt with similarly. Assume that $\lambda_0\notin\partial\sigma_{ess}(J_{0,0})$. Then $\frac{d}{d\lambda} {\rm{Tr}}(M_0(\lambda_0))=0$, otherwise ${\rm{Tr}}(M_0(\lambda))-2$ would change sign at $\lambda_0$ and $\lambda_0$ would separate the elliptic and hyperbolic regions, implying $\lambda_0\in\partial\sigma_{ess}(J_{0,0})$.

We now show that in most cases a diagonal perturbation is sufficient to split the overlapping intervals. Assume that
\begin{equation}\label{p3} |p_3(\lambda_0)|+|p_3'(\lambda_0)|\neq 0.\end{equation}
Let $\lambda$ (depending on $\epsilon$) be a degenerate parabolic point for some $J_{\epsilon,0}$, i.e.
   \begin{equation}\label{ess14}{\rm{Tr}}(M_{\epsilon,0}(\lambda))=2 \quad \hbox{ and }
  \frac{d}{d\lambda}{\rm{Tr}}(M_{\epsilon,0}(\lambda))=0\end{equation}
	Our objective is to show that this cannot happen for sufficiently small $\lambda-\lambda_0$ and $\epsilon$. Due to continuous dependence of the roots on the small parameter $\epsilon$, there is no need to consider the case $\rm{Tr}(M_{\epsilon,0}(\lambda))=-2$. Also, \eqref{p3} will hold (for the same polynomial $p_3$ from $M_0$) with $\lambda_0$ replaced by $\lambda$ in a sufficiently small neighbourhood of $\lambda_0$.
	Noting that $B_{2,\epsilon,0}$ is independent of $\epsilon$ we get that
\begin{align*}
M_{\epsilon,0}&=B_{T}B_{T-1}\dots B_{2,0,0}\bigg(B_{1,0,0}-\frac{\epsilon}{a_1}\left(\begin{array}{cc}
0&0\\
0&1\end{array}\right)\bigg)\\
&=M_0\bigg(I-\frac{\epsilon}{a_1}{B}_{1,0,0}^{-1}\left(\begin{array}{cc}
0&0\\
0&1\end{array}\right)\bigg).
\end{align*}
 As $$B^{-1}_{1,0,0}(\lambda)=\left(\begin{array}{cc}
 \frac{\lambda-b_1}{a_T}&-\frac{a_1}{a_T}\\
 1&0
 \end{array}\right),$$ we get

 \begin{equation}\label{ess16}{\rm{Tr}}(M_{\epsilon,0}(\lambda))=p_1(\lambda)+p_4(\lambda)+\frac{\epsilon p_3(\lambda)}{a_T }.\end{equation}

 Now, Equations~\eqref{ess16} and \eqref{ess14} combine to give the new conditions
 \begin{equation}\label{ess17}
 p_1(\lambda)+p_4(\lambda)+\frac{\epsilon p_3(\lambda)}{a_T }=2
 \end{equation} and
 \begin{equation}\label{ess18}
 p_1'(\lambda)+p_4'(\lambda)+\frac{\epsilon p_3'(\lambda)}{a_T }=0.
 \end{equation}

If Equations~\eqref{ess17} and \eqref{ess18} are both satisfied then we obtain

 \begin{equation}\label{ess20}(2-p_1(\lambda)-p_4(\lambda))p_3'(\lambda)+(p_1'(\lambda)+p_4'(\lambda))p_3(\lambda)=0.\end{equation}

 By invoking Corollary~{\ref{per71}} we observe that the product of polynomials  on the left hand side equals
\begin{equation}\label{ess19} 2p_3'(\lambda)-\left(\frac{p_1(\lambda)+p_4(\lambda)}{p_3(\lambda)}\right)'p_3^2(\lambda).\end{equation}
Note that the term $2p'_3(\lambda)$ is a polynomial of degree not greater than $(T-2)$, the rational function
$\left(\frac{p_1(\lambda)+p_4(\lambda)}{p_3(\lambda)}\right)'$ is of order $0$, and the term $p_3^2(\lambda)$ is a polynomial of degree
$2(T-1)$. Combining these observations we have that the entire last term of the expression is a polynomial of degree $2(T-1)$. Since $2(T-1)$
is greater than $(T-2)$ the whole expression is of degree $2(T-1)$. Clearly, this is not identically zero. Furthermore, this means there are
at most $2T-2$ roots, say $\mu_1,\dots,\mu_{2T-2}$ which are independent of $\epsilon$. Then, since under our assumptions $|p_3(\lambda)|+|p_3'(\lambda)|\neq 0$,
we calculate the valid values for $\epsilon$ by substituting $\lambda:=\mu_i$ into either
Equation~\eqref{ess17} or \eqref{ess18}, and, so, there are at most $2T-2$ valid values for $\epsilon$. In particular, for any sufficiently small $\epsilon\neq 0$, the value $\lambda$ cannot be a degenerate parabolic point for $J_{\epsilon,0}$.
Therefore, all degenerate parabolic points satisfying \eqref{p3} will be split into non-degenerate points for $|\epsilon|\neq 0$ sufficiently small.

It remains to deal with the exceptional case  $p_3(\lambda_0)=p_3'(\lambda_0)=0$. In this case, we use a perturbation with $\epsilon=0$, $\eta\neq 0$. Note that since ${\rm{Tr}}(M_0(\lambda_0))=2$ and $\det\ M_0(\lambda_0)=1$, we have that $p_3(\lambda_0)=0$ implies $p_1(\lambda_0)= p_4(\lambda_0)= 1$
Then
$$B_{1,0,0}^{-1}B_{2,0,0}^{-1}B_{2,0,\eta}B_{1,0,\eta}= \twomat{\frac{a_1}{a_1+\eta}}{\frac{\lambda-b_1}{a_T}\left(\frac{a_1+\eta}{a_1}-\frac{a_1}{a_1+\eta}\right)}{0}{\frac{a_1+\eta}{a_1}}$$
and using \eqref{Meps} we get
$$
{\rm Tr}M_{0,\eta}(\lambda)=p_1(\lambda) \frac{a_1}{a_1+\eta} + p_3(\lambda)\frac{\lambda-b_1}{a_T}\left(\frac{a_1+\eta}{a_1}-\frac{a_1}{a_1+\eta}\right) +p_4(\lambda)\frac{a_1+\eta}{a_1}.
$$
Evaluating at $\lambda_0$, we get
$$
{\rm Tr}M_{0,\eta}(\lambda_0)=\frac{a_1}{a_1+\eta} +\frac{a_1+\eta}{a_1}=\frac{a_1^2+(a_1+\eta)^2}{a_1(a_1+\eta)}=2+\frac{\eta^2}{a_1(a_1+\eta)}>2,
$$
for all $|\eta|\neq 0,$ so $\lambda_0$ is a hyperbolic point for $J_{0,\eta}$ for $\eta\neq 0$. Choosing $|\eta|$ sufficiently small such that no non-degenerate parabolic points can degenerate (see Step One),
this implies that the total degeneracy of the roots at must have decreased by at least one.

Repeating the procedure finitely many times, we can ensure that all roots of ${\rm Tr}M_{\epsilon,\eta}(\lambda)-2$ are simple for sufficiently small non-zero $(\epsilon,\eta)$. Note that in each step, $\epsilon$ or $\eta$ may be chosen arbitrarily small. Since the set of non-degenerate points is open by Step 1, the set of points $(\epsilon,\eta)$ is an open dense set.
\end{proof}

 As a consequence of the above, we obtain the following theorem.
\begin{thm}
For almost all choices of parameters $(a_1,\dots,a_T,b_1,\dots,b_T)\in \left(\mathbb{R}^+\right)^T\times\mathbb{R}^T$ the essential spectrum
(which equals the absolutely continuous spectrum) of the associated Hermitian $T$-periodic Jacobi matrix consists of $T$ distinct real
intervals.
\end{thm}

\section{Solutions to period-T difference equations}\label{sec3}

In this section, and the next, the subsidiary functions our eigenvector will depend upon are defined. From now on until the end of the paper
we will assume for simplicity that $b_i=0$ for all $i$. This restriction is simple to remove.

Firstly, define $B_j(\lambda):=\left(\begin{array}{cc}
0&1\\
-\frac{a_{j-1}}{a_{j}}&\frac{\lambda}{a_{j}}\end{array}\right)$, where $j\in\{1,\dots,T\},$ and
$$M(\lambda):=B_{T}(\lambda)B_{T-1}(\lambda)\dots B_1(\lambda).$$ Then, if $\lambda\in\sigma_{ell}(J_T)$ we have that
$\sigma(M(\lambda))=\{e^{i\theta(\lambda)},e^{-i\theta(\lambda)}\}$ for some real-valued function $\theta(\lambda)$ (the quasi-momentum).
Therefore there exists an invertible matrix $V$ such that $M=V^{-1}\left(\begin{array}{cc}
 \mu&0\\
 0&\overline{\mu}\end{array}\right)V,$ where

 \begin{equation}\label{quasi}
 \mu(\lambda)=e^{i\theta(\lambda)}.
 \end{equation}

\begin{lemma}\label{lem39}
Let $\lambda\in\sigma_{ell}(J_T)$. Then, for any non-zero solution, $(\psi_n)_{n\geq 1}$, to the period-$T$ difference equation,
$a_{n-1}u_{n-1}+a_nu_{n+1}=\lambda u_n, n> 1$, we have the expression
$$(Im(\psi_n))^2=\eta_s(\lambda)\sin(2(k-1)\theta(\lambda)+\phi_s(\lambda))+\gamma_s(\lambda),$$ where $n=T(k-1)+s$, $s\in\{0,\dots,T-1\}$
and $\eta_s,\gamma_s$ are real functions which, along with $\phi_s$, are independent of $k$ and $\theta(\lambda)$ is given by
Equation~\eqref{quasi}.
\end{lemma}

\begin{proof}
Since $\psi_n$ satisfies the difference equation, and $n=T(k-1)+s$ with $s\in\{0,\dots,T-1\}$ we have for any $\left(\begin{array}{c}\psi_0\\
\psi_1\end{array}\right)\in\mathbb{C}^2\setminus\left(\begin{array}{c}0\\ 0 \end{array}\right)$
\begin{align*}
\left(\begin{array}{c}
\psi_n\\
\psi_{n+1}\end{array}\right)&=B_s\dots B_1M^{(k-1)}\left(\begin{array}{c}
\psi_0\\
\psi_1\end{array}\right)\\
&=B_s\dots B_1V\left(\begin{array}{cc}
\mu^{(k-1)}(\lambda)&0\\
0&\overline{\mu}^{(k-1)}(\lambda)\end{array}\right)V^{-1}\left(\begin{array}{c}
\psi_0\\
\psi_1\end{array}\right)\\
&=\left(\begin{array}{c}
\alpha_s(\lambda)e^{i(k-1)\theta(\lambda)}+\beta_s(\lambda)e^{-i(k-1)\theta(\lambda)}\\
\kappa_s(\lambda)e^{i(k-1)\theta(\lambda)}+\chi_s(\lambda)e^{-i(k-1)\theta(\lambda)}\end{array}\right),
\end{align*}
for some functions $\alpha_s,\beta_s,\kappa_s,\chi_s$ of $\lambda\in\sigma_{ell}(J_T)$ and $s$. In the case of $s=0$ we interpret $B_0\dots
B_1$ to equal the identity, and $B_1\dots B_1=B_1$. Consequently,
$$\psi_n=\psi_{T(k-1)+s}={\alpha_s}(\lambda)e^{i(k-1)\theta(\lambda)}+{\beta_s}(\lambda)e^{-i(k-1)\theta(\lambda)}.$$ Thus,
\begin{equation}\label{172}
Im(\psi_n)=\widetilde{\alpha_s}(\lambda)\sin((k-1)\theta)+\widetilde{\beta_s}(\lambda)\cos((k-1)\theta),
\end{equation}
where $\widetilde{\alpha}_s(\lambda):=Re(\alpha_s(\lambda))-Re(\beta_s(\lambda))$,
$\widetilde{\beta}_s(\lambda):=Im(\alpha_s(\lambda))+Im(\beta_s(\lambda))$ are real-valued functions of $\lambda$.

Furthermore, using the double-angle formulae, $\sin(2x)=2\sin(x)\cos(x)$ and $\cos(2x)=\cos^2(x)-\sin^2(x)$, we have
\begin{align*}
\left(Im(\psi_n)\right)^2&=\widetilde{\alpha_s}^2\sin^2((k-1)\theta)+\widetilde{\beta_s}^2\cos^2((k-1)\theta)\\
&\hspace{10pt}+\widetilde{\alpha_s}\widetilde{\beta_s}\left(2\sin((k-1)\theta)\cos((k-1)\theta)\right)\\
&=\eta_s\sin(2(k-1)\theta+\phi_s)+\gamma_s,
\end{align*}
where $\eta_s,\phi_s$ and $\gamma_s$ are real-valued.
\end{proof}

\begin{remark}
Clearly, by suitably choosing  $\psi_0,\psi_1$ the vector $(\eta_0,\dots,\eta_{T-1})$ can be arranged to be non-trivial.
\end{remark}

\begin{lemma}\label{per31}
Let $\lambda\in\sigma_{ell}(J_T)$. Then, there exists a particular non-zero solution, $(\varphi_n)_{n\geq1}$, to the period-$T$ difference
equation $a_{n-1}u_{n-1}+a_{n}u_{n+1}=\lambda u_n, n>1$ which has the property \begin{equation}\label{4.6}
\varphi_n(\lambda)={\varphi}_s(\lambda) e^{i(k-1)\theta(\lambda)}\end{equation} for some non-trivial set of functions
$({\varphi}_s)_{s=0}^{T-1}$, where $n=T(k-1)+s, s\in\{0,\dots,T-1\}$.
\end{lemma}

\begin{proof}
Recall from above that $M(\lambda)$ has eigenvalues $e^{\pm i\theta(\lambda)},$ i.e. $$M(\lambda)\left(\begin{array}{c}
\varphi_{0}\\
\varphi_{1}\end{array}\right)=e^{i\theta(\lambda)}\left(\begin{array}{c}
\varphi_{0}\\
\varphi_{1}\end{array}\right),$$ for some $\varphi_0,\varphi_1$. Define $\varphi_{2},\dots,\varphi_{T-1}$ by $$\left(\begin{array}{c}
{\varphi}_{s}\\
{\varphi}_{s+1}\end{array}\right):=B_{s}B_{s-1}\dots B_1\left(\begin{array}{c}
\varphi_{0}\\
\varphi_{1}\end{array}\right). $$ Then using the notation $n=T(k-1)+s$,

\begin{align*}
\left(\begin{array}{c}
\varphi_{n}\\
\varphi_{n+1}\end{array}\right)&=B_{s}\dots B_1 M^{(k-1)}\left(\begin{array}{c}
\varphi_{0}\\
\varphi_{1}\end{array}\right)\\
&=B_{s}\dots B_1 e^{(k-1)i\theta(\lambda)}\left(\begin{array}{c}
\varphi_{0}\\
\varphi_{1}\end{array}\right)\\
&=e^{(k-1)i\theta(\lambda)}\left(\begin{array}{c}
{\varphi}_{s}\\
{\varphi}_{s+1}\end{array}\right).
\end{align*}
Consequently, $$\varphi_n={\varphi}_s e^{(k-1)i\theta(\lambda)}.\qedhere$$
\end{proof}

\begin{remark}
Henceforth, the eigenvector of the monodromy matrix will be normalized with $\varphi_0=1$. Subsequent calculations in Lemma~\ref{per32} will
confirm the validity of this choice for almost every $\lambda$.
\end{remark}

\section{Properties of the function $C(\lambda;T)$}
 In this section we introduce a new, analytic function of $\lambda$, $C(\lambda;T)$. This will play an important role in the asymptotic
 expansion of our eigenvector, $(u_n)$. Its zeros will give values of $\lambda$ where our construction fails. Here we explore its properties
 and structure.
 \begin{mydef}\label{c4.1}
 For $\lambda\in\sigma_{ell}(J_T)$, let
 $C(\lambda;T):=Re\left(\sum\limits_{s=1}^{T}{\varphi}_s(\lambda)\overline{\varphi}_{s-1}(\lambda)\right),$ where ${\varphi}_s$ are as in
 Lemma~\ref{per31}.
 \end{mydef}

Note that Definition~\ref{c4.1}~is invariant w.r.t. the choice of branches $\mu$ and $\overline{\mu}$ on $\sigma_{ell}(J_T)$. Indeed, since
$\lambda\in\sigma_{ell}(J_T)\subset\mathbb{R}$, all matrix elements of $B_s(\lambda),s=1,2,\dots,T$ and $M(\lambda)$ are real polynomials,
$\varphi_0=1$ and $$\varphi_1(\lambda)=\left(\mu-m_{11}(\lambda)\right)m_{12}^{-1}(\lambda)$$ changes under the transformation $\mu\mapsto
\overline{\mu}$ to the complex conjugate function $\varphi_{1}(\lambda)\mapsto\overline{\varphi}_1(\lambda),\lambda\in\sigma_{ell}(J_T)$. For
the last fact the inclusion $\lambda\in\sigma_{ell}(J_T)$ is essential. Hence for all $s=1,\dots,T$
$\varphi_s(\lambda)\mapsto\overline{\varphi}_s(\lambda)$ and the expression for $C(\lambda;T)$, $\lambda\in\sigma_{ell}(J_T)$, transforms
into
\begin{align*}
Re\left(\sum\limits_{i=1}^T \overline{\varphi}_s(\lambda)\varphi_{s-1}(\lambda)\right)\ =\ Re\left(\overline{\sum\limits_{i=1}^T
{\varphi_s(\lambda)}\overline{\varphi}_{s-1}(\lambda)}\right) \ =\ C(\lambda;T).
\end{align*}

\begin{mydef}
Consider a rational function in the variable $x$ of the form $\frac{P(x)}{Q(x)}$ where $P(x), Q(x)$ are polynomials. The order of the rational function is defined to be the difference in degree of the polynomials $P(x)$ and $Q(x)$: $\deg P - \deg Q$.
\end{mydef}

\begin{lemma}\label{per32}
The function $C(\lambda;T)$ is a rational function on $\sigma_{ell}(J_T)$ and can be extended, uniquely, as a rational function to
$\mathbb{C}$.
\end{lemma}

\begin{proof}

Let $\lambda\in\sigma_{ell}(J_T)$. First, the special cases of $T=1$ and $T=2$ must be considered separately. For $T=1$, ${\varphi}_0=1$ and
${\varphi}_1=\mu$, so
\begin{align*}
C(\lambda;1)&=Re\left({\varphi}_1\overline{{\varphi}}_0\right)=Re(\mu)=\frac{{\rm Tr}(M(\lambda))}{2}=\frac{\lambda}{2a_1}.
\end{align*}

For $T=2$ we have $$M(\lambda)=\left(\begin{array}{cc}
-\frac{a_2}{a_1}&\frac{\lambda}{a_1}\\
-\frac{\lambda}{a_1}&\frac{\lambda^2}{a_1a_2}-\frac{a_1}{a_2}
\end{array}\right).$$ By defining $\varphi_0:=1, \varphi_1$ is such that $$M(\lambda)\left(\begin{array}{c}
\varphi_{0}\\
\varphi_{1}\end{array}\right)=e^{i\theta(\lambda)}\left(\begin{array}{c}
\varphi_{0}\\
\varphi_{1}\end{array}\right).$$ Consequently, $$\varphi_1=\frac{a_1\mu+a_2}{\lambda}, \varphi_2=\mu.$$ Then, using $\mu\overline{\mu}=1,$
\begin{align}
C(\lambda;2)&=Re\left({\varphi}_2\overline{{\varphi}}_1+{\varphi}_1\overline{{\varphi}}_{0}\right)\label{cl2}\\
&=Re\left(\mu\frac{a_1\overline{\mu}+a_2}{\lambda}+\frac{a_1\mu+a_2}{\lambda}\right)\nonumber\\
&=\frac{(a_1+a_2)}{\lambda}\left(1+Re(\mu)\right)=\frac{(a_1+a_2)}{\lambda}\left(1+\frac{{\rm Tr}(M(\lambda))}{2}\right)\nonumber\\
&=\frac{(a_1+a_2)}{2\lambda a_1a_2}\left(\lambda^2-(a_1^2+a_2^2)+2a_1a_2\right)=\frac{(a_1+a_2)}{2\lambda
a_1a_2}\left(\lambda^2-|a_1-a_2|^2\right).\nonumber
\end{align}
   Thus, the assertion holds for both of these cases.

For $T\geq 3$ we define $\varphi_0:=1$ and follow a similar technique to the case for $T=2$. Here we see that the normalisation $\varphi_0=1$ is
valid unless $m_{12}(\lambda)=0$. Consequently, for $m_{12}(\lambda)\neq 0,$ $$\varphi_1=\frac{\mu-m_{11}}{m_{12}},$$ where $m_{11},m_{12}$
are as described in Corollary~\ref{per71}. Throughout the proof $P_k$ will denote a polynomial of at most degree $k$, while
$R_k,\widetilde{R}_k$ will denote rational functions of order at most $k$. Using Lemma~\ref{per71}, again, and a similar calculation as in the case of Lemma~\ref{per31}, for $s=2,\dots,T$ we obtain
$$\varphi_s=\left(\frac{\lambda^{s-1}}{\prod_{j=1}^{s-1}
a_j}+~P_{s-2}~\right)\frac{\mu-m_{11}}{m_{12}}+\left(-\frac{a_T\lambda^{s-2}}{\prod_{j=1}^{s-1} a_j}+~{P}_{s-3}~\right),$$
 where $P_{-1}$ and $\widetilde{P}_{-1}$ are both identically zero. Note that $Im(\mu)$ is an algebraic but not rational function of
$\lambda$. Indeed, ${\rm Tr}(M(\lambda))$ is a polynomial in $\lambda$ and $det(M(\lambda))=1$, therefore $Im(\mu)$ is the root of
$\left(\frac{{\rm Tr}(M(\lambda))}{2}\right)^2+1$, which would be the square of a rational function iff ${\rm Tr}\left(M(\lambda)\right)$ were equal to a
constant. However, since $\mu\overline{\mu}=1$ and $$Re(\mu)=Re(\overline{\mu})=\frac{{\rm Tr}(M(\lambda))}{2},$$
$Re\left(\varphi_s(\lambda)\overline{\varphi}_{s-1}(\lambda)\right)$ is clearly a rational function of $\lambda$, $C(\lambda;T)$ is also a
rational function of $\lambda$, $\lambda\in\sigma_{ell}(J_T)$. Now we see that $C(\lambda;T)$ is well-defined as an analytic function not
only on $\sigma_{ell}(J_T)$, but everywhere on $\mathbb{C}$ except at the roots of $m_{12}(\lambda)$.
\end{proof}

\begin{remark}  The function $C(\lambda;T)$ only fails to be defined when the polynomial $m_{12}(\lambda)$, defined in Corollary~\ref{per71},
is equal to 0. For $\lambda\in\sigma_{ell}(J_T)$ we have $m_{12}(\lambda)\neq 0$ since if $m_{12}(\lambda)=0$ then the eigenvalues of the
monodromy matrix for real $\lambda$ are real, and as usual their product is $1$. Indeed, since $m_{12}(\lambda)=0$, we have that the
monodromy matrix is lower-triangular and therefore $m_{11}(\lambda)$ and $ m_{22}(\lambda)$ are the (real) eigenvalues. Thus, $\lambda$ is either
in the hyperbolic or parabolic case, contradicting that $\lambda\in\sigma_{ell}(J_T)$, and so the denominator has no roots in
$\sigma_{ell}(J_T)$.
\end{remark}

Our technique for embedding eigenvalues fails for values $\lambda$ when the function $C(\lambda;T)=0$. It is important to understand when
this situation arises.

\begin{remark}
For the case $T=1$, the function $C(\lambda;1)$ has only one root at $\lambda=0$.
From Equation~\eqref{cl2}~ we know that for the case $T=2$ the function $C(\lambda;2)$ has no zeros for $\lambda\in\sigma_{ell}(J_2)$ as its
two roots, $\lambda_{\pm}=\pm|a_1-a_2|$, are parabolic points.
For the case $T=3$ the function
\begin{equation}\label{4.21}C(\lambda;3)=\frac{\lambda\left(\frac{1}{a_1}+\frac{1}{a_2}+\frac{1}{a_3}\right)}{2(\lambda^2-a_1^2)}\left(\lambda^2-(a_1^2+a_2^2+a_3^2)+\frac{2(a_1+a_2+a_3)}{\frac{1}{a_1}+\frac{1}{a_2}+\frac{1}{a_3}}\right)\end{equation}
has a zero at $\lambda=0$. In order to preclude any other roots in the generalised interior of the a.c. spectrum (see Definition \ref{genint}) it is sufficient to
establish that $|{\rm Tr}(M(\lambda))|\geq 2$ whenever $C(\lambda;3)=0$. A simple calculation shows that this is equivalent to
\begin{multline*}g(a_1,a_2):=(a_1^3+a_1^3a_2+a_2^3+a_1a_2^3+a_2+a_1-a_1^2a_2-a_1a_2^2-a_1a_2)(a_1+a_2+1)^2\\-(a_1+a_2+a_1a_2)^3\geq 0, \end{multline*} where, by homogeneity, w.l.o.g~$a_3=1$. Numerical calculations of the roots of $g$ suggest that this function is non-negative for $a_1,a_2>0$. More generally, we believe that for even $T$ the function $C(\lambda;T)$ has no zeros in the
generalized interior of the a.c. spectrum, and for odd $T$ there is a single solution at $\lambda=0$.
\end{remark}

\begin{remark}
Now consider a different formula for $C(\lambda;T)$ having a ``symplectic character". Using it one can easily deduce, in a slightly different
way, the rationality of $C(\lambda;T)$. Introducing the indefinite matrix $\hat{\mathcal{J}}:=\left(\begin{array}{cc}
0&1\\
1&0 \end{array}\right)$ in $\mathbb{C}^2$ one can rewrite the expression for $C(\lambda;T)$ in the following form (we assume below that
$\lambda\in\sigma_{ell}(J_T)$):
\begin{align*}
C(\lambda;T)&=\frac{1}{2}\sum\limits_{s=1}^T
\left(\varphi_s(\lambda)\overline{\varphi}_{s-1}(\lambda)+\varphi_{s-1}(\lambda)\overline{\varphi}_s(\lambda)\right)\\
&\hspace{-10pt}=\frac{1}{2}\sum\limits_{s=1}^T\left\langle \hat{\mathcal{J}}\left(\prod\limits_{k=1}^{s-1}B_k(\lambda)\right)\left(\begin{array}{c}
\varphi_0\\
\varphi_1(\lambda)\end{array}\right),\left(\prod\limits_{k=1}^{s-1}B_k(\lambda)\right)\left(\begin{array}{c}
\varphi_0\\
\varphi_1(\lambda) \end{array}\right)\right\rangle_{\mathbb{C}^2}\\
&\hspace{-10pt}=\frac{1}{2}\sum\limits_{s=1}^T\left\langle
\left(\prod\limits_{k=1}^{s-1}B_k(\overline{\lambda})\right)^*\hat{\mathcal{J}}\left(\prod\limits_{k=1}^{s-1}B_k(\lambda)\right)\left(\begin{array}{c}
\varphi_0\\
\varphi_1(\lambda)\end{array}\right),\left(\begin{array}{c}
\varphi_0\\
\varphi_1(\overline{\lambda})\end{array}\right)\right\rangle_{\mathbb{C}^2}\\
&\hspace{-10pt}=\sum\limits_{s=1}^T\bigg\langle F_s(\lambda)\left[\left(\begin{array}{c}
1\\
-m_{11}(\lambda)m_{12}^{-1}(\lambda)\end{array}\right)+\mu(\lambda)\left(\begin{array}{c}
0\\
m_{12}^{-1}(\lambda)\end{array}\right)\right],\\
&\hspace{60pt}\left[\left(\begin{array}{c}
1\\
-m_{11}(\overline{\lambda})m_{12}^{-1}(\overline{\lambda})\end{array}\right)+\mu(\overline{\lambda})\left(\begin{array}{c}
0\\
m_{12}^{-1}(\overline{\lambda})\end{array}\right)\right]\bigg\rangle_{\mathbb{C}^2}
\end{align*}
where we denoted the real matrix polynomials
$$\frac{1}{2}\left(\prod\limits_{k=1}^{s-1}B_k(\overline{\lambda})\right)^*\hat{\mathcal{J}}\left(\prod\limits_{k=1}^{s-1}B_k(\lambda)\right)$$
by $F_s(\lambda)$, $s>1$ and $F_1(\lambda):=\frac{\hat{\mathcal{J}}}{2}$. Therefore
\begin{align}
C(\lambda;T)&=\sum\limits_{s=1}^T\bigg\{\left\langle F_s(\lambda)\left(\begin{array}{c}
1\\
-m_{11}(\lambda)m_{12}^{-1}(\lambda)\end{array}\right),\left(\begin{array}{c}
1\\
-m_{11}(\overline{\lambda})m_{12}^{-1}(\overline{\lambda})\end{array}\right)\right\rangle_{\mathbb{C}^2}\nonumber\\
&~~+{\rm Tr}\left(M(\lambda)\right)\left\langle F_s(\lambda)\left(\begin{array}{c}
1\\
-m_{11}(\lambda)m_{12}^{-1}(\lambda)\end{array}\right),\left(\begin{array}{c}
1\\
m_{12}^{-1}(\overline{\lambda})\end{array}\right)\right\rangle_{\mathbb{C}^2}\nonumber\\
&~~~+\left\langle F_s(\lambda)\left(\begin{array}{c}
1\\
m_{12}^{-1}(\lambda)\end{array}\right),\left(\begin{array}{c}
1\\
m_{12}^{-1}(\overline{\lambda})\end{array}\right)\right\rangle_{\mathbb{C}^2}\bigg\}\label{4.35}
\end{align}
where we used that ${\rm Tr}(M(\lambda))=\mu(\lambda)+{\overline{\mu}}(\lambda),\lambda\in\sigma_{ell}(J_T).$
From the last expression, taking into consideration that $F_s(\lambda), m_{11}(\lambda), m_{12}(\lambda), {\rm Tr}\left(M(\lambda)\right)$ are
polynomials in $\lambda$, we see immediately that $C(\lambda;T)$ is a rational function of $\lambda$ on $\sigma_{ell}(J_T)$ and therefore
admits unique analytic continuation as a rational function to the whole of $\mathbb{C}$, given by Formula~\eqref{4.35}. Moreover, using the
last formula one can give an upper bound for the order of $C(\lambda;T)$ as a rational function, but in the next theorem we will present an
explicit calculation of the order.
\end{remark}

\begin{thm}
The function $C(\lambda;T)$ is a rational function of $\lambda$ of order $1$. Moreover, its asymptotic expansion is given by
$$C(\lambda;T)\sim \frac{1}{2}\left(a_1^{-1}+\dots+a_T^{-1}\right)\lambda, \lambda\rightarrow\infty.$$
\end{thm}

\begin{proof}
{\underline{(Step One)}} For $\lambda\in\sigma_{ell}(J_T)$ we have
\begin{align*}
C(\lambda;T)&=\sum\limits_{s=1}^T
\varphi_s(\lambda){\overline{\varphi}}_{s-1}(\lambda)-iIm\left(\sum\limits_{s=1}^T\varphi_s(\lambda){\overline{\varphi}}_{s-1}(\lambda)\right)\\
&=\sum\limits_{s=1}^T
\varphi_s(\lambda){\overline{\varphi}}_{s-1}(\lambda)-\frac{1}{2}\sum\limits_{s=1}^T\left(\varphi_s(\lambda){\overline{\varphi}}_{s-1}(\lambda)-\varphi_{s-1}(\lambda){\overline{\varphi}}_{s}(\lambda)\right)
\end{align*}
Note that by the constancy of the ``discrete Wronskian" we know that (using $a_0:=a_T$)
\begin{equation}\label{5.27}a_s(\varphi_{s+1}(\lambda){\overline{\varphi}}_s(\lambda)-\varphi_s(\lambda){\overline{\varphi}}_{s+1}(\lambda))=a_{s-1}(\varphi_{s}(\lambda){\overline{\varphi}}_{s-1}(\lambda)-\varphi_{s-1}(\lambda){\overline{\varphi}}_{s}(\lambda)),\end{equation}
$s=1,2,\dots,T-1$. The last identity can be easily proved using the fact that both $\varphi_s(\lambda)$ and ${\overline{\varphi}}_s(\lambda)$
are solutions to the recurrence relations. Applying Equation~\eqref{5.27} one obtains

\begin{eqnarray}\label{5.28}
C(\lambda;T)&=&\left(\sum\limits_{s=1}^T\varphi_s(\lambda){\overline{\varphi}}_{s-1}(\lambda)\right)\nonumber\\
&&-\frac{1}{2}\left(a_1^{-1}+\dots+a_T^{-1}\right)a_T\left(\varphi_1(\lambda){\overline{\varphi}}_0(\lambda)-\varphi_0(\lambda){\overline{\varphi}}_1(\lambda)\right)
\end{eqnarray} since
\begin{align*}
\varphi_1(\lambda){\overline{\varphi}}_0(\lambda)-\varphi_0(\lambda){\overline{\varphi}}_1(\lambda)&=\varphi_1(\lambda)-{\overline{\varphi_1}}(\lambda)\\
&\hspace{-30pt}=\left(\mu(\lambda)-m_{11}(\lambda)m_{12}^{-1}(\lambda)\right)-\left({\overline{\mu}}(\lambda)-m_{11}(\lambda)m_{12}^{-1}(\lambda)\right)\\
&\hspace{-30pt}=\left(\mu(\lambda)-{\overline{\mu}}(\lambda)\right)m_{12}^{-1}(\lambda)
\end{align*}
admits analytic continuation from $\sigma_{ell}(J_T)$ to $\mathbb{C}\setminus\sigma_{ess}(J_T)$ as an analytic (algebraic, but not rational)
function $\left(\mu(\lambda)-{{\mu^{-1}}}(\lambda)\right)m_{12}^{-1}(\lambda)$. It asymptotically behaves like
$$-{\rm Tr}(M(\lambda))m_{12}^{-1}(\lambda)\sim -\lambda^T\left(\prod\limits_{s=1}^T
a_s\right)^{-1}\left(\lambda^{T-1}\left(\prod\limits_{s=1}^{T-1} a_s\right)^{-1}\right)^{-1}=-\frac{\lambda}{a_T}$$
assuming that the branch of the analytic function $\mu(\lambda)$ has been chosen so that $\mu(\lambda)\rightarrow 0$, as
$\lambda\rightarrow\infty$. Note that in $\mathbb{C}^+\cup\mathbb{C}^-$ we are in the hyperbolic situation (Lemma~\ref{4.20}) so the
eigenvalues of $M(\lambda)$ (which are $\mu(\lambda),\mu^{-1}(\lambda)$ as $det(M(\lambda))\equiv1$) with one of them (at our choice)
behaving at infinity like $\mu(\lambda)\sim \left(\lambda^T/\left(\prod\limits_{s=1}^T a_s\right)\right)^{-1}$ and the other one like
$\mu^{-1}(\lambda)\sim \left(\lambda^T/\prod\limits_{s=1}^T a_s\right)$. Note that the correct choice of the branch of $\mu(\lambda)$
(despite invariance of the definition of $C(\lambda;T)$ under that choice) is crucial for our proof.

Therefore the second term in Equation~\eqref{5.28} admits the asymptotics $$\frac{1}{2}\left(a_1^{-1}+\dots+a_T^{-1}\right)\lambda$$ as
$\lambda\rightarrow\infty$ according to our choice of the branch $\mu(\lambda).$ The opposite choice of the branch changes the sign in the
above mentioned asymptotics of the second term and therefore leads to a sophisticated calculation of the first term, which we are not able to
produce here.

{\underline{(Step Two)}} We next analyse the asymptotics at infinity of the first term in Formula~\eqref{5.28}. Since $\varphi_0=1,$
$$\varphi_1(\lambda)=(\mu(\lambda)-m_{11}(\lambda))m_{12}^{-1}(\lambda)=O\left(\lambda^{-1}\right),$$ as
$\mu(\lambda)=O\left(\lambda^{-T}\right),\lambda\rightarrow\infty$. For the function $\varphi_s(\lambda)$ we have
\begin{align*}
\left(\begin{array}{c}
\varphi_s(\lambda)\\
\varphi_{s+1}(\lambda)\end{array}\right)&=B_s(\lambda)\dots B_1(\lambda)\left(\begin{array}{c}
\varphi_0\\
\varphi_{1}(\lambda)\end{array}\right)\\
&=\left(B_{s+1}^{-1}(\lambda)\dots B_{T}^{-1}(\lambda)\right)M(\lambda)\left(\begin{array}{c}
\varphi_0\\
\varphi_{1}(\lambda)\end{array}\right)\\
&=\mu\left(B_{s+1}^{-1}(\lambda)\dots B_{T}^{-1}(\lambda)\right)\left(\begin{array}{c}
\varphi_0\\
\varphi_{1}(\lambda)\end{array}\right)\\
&=\left[\mu\left(B_{s+1}^{-1}(\lambda)\dots B_{T}^{-1}(\lambda)\right)\right]\left(\begin{array}{c}
1\\
O\left(\frac{1}{\lambda}\right)\end{array}\right)\\
&=\left(\begin{array}{c}
O\left(\lambda^{(T-s)-T}\right)\\
O\left(\lambda^{(T-s)-T}\right)\end{array}\right)=O\left(\lambda^{-s}\right)
\end{align*}
since obviously the matrix function $$\mu(\lambda)B_{s+1}^{-1}(\lambda)\dots B_{T}^{-1}(\lambda)=O\left(\lambda^{(T-s)-T}\right),$$ as
$$B_s^{-1}(\lambda)=\frac{a_s}{a_{s-1}}\left(\begin{array}{cc}
\frac{\lambda}{a_s}&-1\\
\frac{a_{s-1}}{a_s}&0\end{array}\right)=O\left(\lambda\right)$$
and $\mu(\lambda)=O\left(\lambda^{-T}\right),\lambda\rightarrow\infty$. Hence
$$\varphi_s(\lambda)=O\left(\lambda^{-s}\right),\lambda\rightarrow\infty,$$ $s=1,2,\dots,T$.

{\underline{(Step Three)}} We now analyse the asymptotics of the analytic continuation of $\overline{\varphi}_{s-1}(\lambda)\equiv
\overline{\varphi_{s-1}(\overline{\lambda})}$. Taking the complex conjugate for $\lambda\in\sigma_{ell}(J_T)$ we get
$$\left(\begin{array}{c}
\overline{\varphi}_s(\overline{\lambda})\\
\overline{\varphi}_{s+1}(\overline{\lambda})\end{array}\right)=B_s(\lambda)\dots B_1(\lambda)\left(\begin{array}{c}
1\\
\left(\mu^{-1}(\lambda)-m_{11}(\lambda)\right)m_{12}^{-1}(\lambda)\end{array}\right).$$
Since $\mu^{-1}(\lambda)=O\left(\lambda^T\right),\lambda\rightarrow\infty,$ for the analytic continuation to $\mathbb{C}$ this gives
$$\left(\begin{array}{c}
\overline{\varphi}_s(\overline{\lambda})\\
\overline{\varphi}_{s+1}(\overline{\lambda})\end{array}\right)=B_s(\lambda)\dots B_1(\lambda)\left(\begin{array}{c}
1\\
O\left(\lambda\right)\end{array}\right),$$
where the matrix polynomial $B_s(\lambda)\dots B_1(\lambda)=O\left(\lambda^s\right),\lambda\rightarrow\infty$. So,
$$\overline{\varphi_{s+1}(\overline{\lambda})}=O\left(\lambda^{s+1}\right),$$ $\lambda\rightarrow\infty, s=0,1,\dots,T-1.$

{\underline{(Step Four)}} Combining both asymptotic formulas for $\varphi_s(\lambda)$ and $\overline{\varphi}_s(\overline{\lambda})$ we
finally obtain
\begin{align*}
\sum\limits_{s=1}^T\varphi_s(\lambda)\overline{\varphi}_{s-1}(\overline{\lambda})&=\sum\limits_{s=1}^T O\left(\lambda^{-s}\right)\cdot
O\left(\lambda^{s-1}\right)\\
&=\sum\limits_{s=1}^T O\left(\lambda^{-1}\right)=O\left(\lambda^{-1}\right)
\end{align*}
as $\lambda\rightarrow\infty$, which leads to the formula
$$C(\lambda;T)=\frac{1}{2}\left(a_1^{-1}+\dots+a_T^{-1}\right)\lambda+O\left(1\right),\lambda\rightarrow\infty.$$ As a corollary we obtain
that the function $C(\lambda;T)$ is always of order exactly $1$ and is therefore never identically zero.
\end{proof}

\section{The ansatz for the eigenvector and its asymptotics}

In this section we plan to elaborate on the explicit construction of the eigenvector associated with the eigenvalue embedded in the a.c.
spectrum of the Jacobi matrix with a diagonal perturbation of Coulomb-type decay.

The following classical result will be used in the next lemma.

\begin{prop}\label{zyg}{\rm{(see~\cite{5})}}. Assume $\alpha,\gamma\in\mathbb{R},\gamma>0$, then the following estimate holds:
$$\sum\limits_{k=n}^\infty \frac{e^{ik\alpha}}{k^\gamma}=O\left(1/n^\gamma\right),n\rightarrow\infty,~~~\iff
\frac{\alpha}{2\pi}\not\in\mathbb{Z}.$$
\end{prop}

We will now introduce the function $\omega_n$ which is an important part of the eigenvector of the embedded eigenvalue.

\begin{lemma}\label{per36}
Let $\lambda\in\sigma_{ell}(J_T),\alpha>1$ and  \begin{equation}\label{4.5}\omega_n(\lambda):=\sum\limits_{m=n+1}^\infty
m^{-\alpha}Im(\varphi_m(\lambda))Im(\varphi_{m-1}(\lambda)),\end{equation} where $(\varphi_n)$ is defined as in \eqref{4.6}. Then
\begin{equation}\label{asyp1} \omega_n=\frac{C(\lambda;T)}{2(\alpha-1)T n^{\alpha-1}}+O\left(\frac{1}{n^\alpha}\right),
n\rightarrow\infty.\end{equation}  Moreover, $\omega_n\in l^2$ for $\alpha>\frac{3}{2}$.
\end{lemma}
\begin{remark}
Formula~\eqref{asyp1} shows that at zeros of $C(\lambda;T)$ the asymptotics for the function $\omega_n$ change drastically. This proves the
importance of our analysis in Section 4.
\end{remark}

\begin{proof}
 The proof is divided into two cases.

{\underline{Case 1}} If $n=T(k-1)$ then by Lemma~\ref{per31} we obtain the relation
\begin{align*}
\omega_n&=\sum\limits_{j=k-1}^\infty \sum\limits_{s=1}^{T}(Tj+s)^{-\alpha}Im\left(\varphi_{Tj+s}\right)Im\left(\varphi_{Tj+s-1}\right)\\
&=\sum\limits_{j=k-1}^\infty
(Tj)^{-\alpha}\left(\sum\limits_{s=1}^{T-1}\left[Im\left(e^{ij\theta}{\varphi}_s\right)Im\left(e^{ij\theta}{\varphi}_{s-1}\right)\right]\right.\\
&\hspace{90pt}+Im\left(e^{i(j+1)\theta}{\varphi}_0\right)Im\left(e^{ij\theta}{\varphi}_{T-1}\right)\bigg)+O(k^{-\alpha})\\
&=\sum\limits_{j=k-1}^\infty
(Tj)^{-\alpha}\bigg(\left[\sum\limits_{s=1}^{T-1}\frac{-1}{4}\left(e^{ij\theta}{\varphi}_s-e^{-ij\theta}\overline{{\varphi}}_s\right)\left(e^{ij\theta}{\varphi}_{s-1}-e^{-ij\theta}\overline{{\varphi}}_{s-1}\right)\right]\\
&~~~~~-\frac{1}{4}\left(e^{i(j+1)\theta}{\varphi}_0-e^{-i(j+1)\theta}\overline{{\varphi}}_0\right)\left(e^{ij\theta}{\varphi}_{T-1}-e^{-ij\theta}\overline{{\varphi}}_{T-1}\right)\bigg)+O(k^{-\alpha}).
\end{align*}
Then, $\theta(\lambda)\not\in\pi\mathbb{Z}$ as $\lambda\in\sigma_{ell}(J_T)$, so by Proposition~\ref{zyg}
\begin{align*}
\omega_n&=\frac{T^{-\alpha}}{4}\sum\limits_{j=k-1}^\infty
\frac{1}{j^{\alpha}}\left(\sum\limits_{s=1}^{T-1}\left({\varphi}_s\overline{{\varphi}}_{s-1}+\overline{{\varphi}}_s{\varphi}_{s-1}\right)+e^{i\theta}{\varphi}_0\overline{{\varphi}}_{T-1}+e^{-i\theta}\overline{{\varphi}}_0{\varphi}_{T-1}\right)\\
&\hspace{20pt}+O(k^{-\alpha})\\
&=T^{-\alpha}\sum\limits_{j=k-1}^\infty \frac{j^{-\alpha}}{2}  C(\lambda;T)+O(k^{-\alpha}).
\end{align*}
 Thus we can apply the Integral Test and obtain
$$\omega_n=\frac{C(\lambda;T)}{2(\alpha-1)T n^{\alpha-1}}+O\left(\frac{1}{n^\alpha}\right).$$
Finally, if $C(\lambda;T)\neq 0$,

$$\omega_n\asymp n^{1-\alpha}\in l^2\iff \alpha>\frac{3}{2}.$$
 This proves the result for Case 1.

{\underline{Case 2}} If $n=T(k-1)+s_n$ with $s_n\in\{1,\dots,T-1\}$. Then

$$\omega_n=\sum\limits_{j=k}^\infty \sum\limits_{s=1}^{T}(Tj+s)^{-\alpha}Im\left(\varphi_{Tj+s}\right)Im\left(\varphi_{Tj+s-1}\right)+F(n)$$
where, noting that $s_n+1\geq2,$
\begin{align*}
F(n)&:=\sum\limits_{s=s_n+1}^{T}\left(T(k-1)+s\right)^{-\alpha} Im\left(\varphi_{T(k-1)+s}\right) Im\left(\varphi_{T(k-1)+s-1}\right)\\
&=\sum\limits_{s=s_n+1}^{T}\left(T(k-1)+s\right)^{-\alpha}
Im\left(e^{i(k-1)\theta}\widetilde{\varphi}_s\right)Im\left(e^{i(k-1)\theta}\widetilde{\varphi}_{s-1}\right)\\
&=O(k^{-\alpha})=O\left(n^{-\alpha}\right).
\end{align*}
Thus, the remainder can be absorbed in the error term.
\end{proof}

We now make an ansatz for the eigenvector of the embedded eigenvalue, $\lambda\in\sigma_{ell}(J_T)$, in the form
$$u_n=Im(\varphi_n)\omega_n.$$

\begin{thm}
The sequence, $(u_n)$, has the asymptotic form $$u_n=\frac{\widetilde{\eta}_s\sin(n{\theta/
T}+\widetilde{\zeta}_s)}{n^{\alpha-1}}+O\left(\frac{1}{n^\alpha}\right),$$ where $\widetilde{\eta_s}$ and $\widetilde{\zeta}_s$ are real
functions, $\alpha>1, n=T(k-1)+s$ with $s\in\{0,\dots, T-1\}$ and $\theta(\lambda)$ as in Equation~\eqref{quasi}. Moreover, the vector $(\widetilde{\eta}_s)_{s=0}^{T-1}$ is equal to the product of $C(\lambda;T)$ with some non-zero vector. Therefore the only
source of vanishing leading terms in the function $u_n$ is the vanishing of $C(\lambda;T)$.
\end{thm}

\begin{proof}
By Equation~\eqref{172}
\begin{align*}
Im(\varphi_n)&=\widetilde{\alpha}_s\sin((k-1)\theta)+\widetilde{\beta}_s\cos((k-1)\theta)\\
&=\sqrt{\widetilde{\alpha}^2_s+\widetilde{\beta}^2_s}\left(\frac{\widetilde{\alpha}_s}{\sqrt{\widetilde{\alpha}^2_s+\widetilde{\beta}^2_s}}\sin((k-1)\theta)+\frac{\widetilde{\beta}_s}{\sqrt{\widetilde{\alpha}^2_s+\widetilde{\beta}^2_s}}\cos((k-1)\theta)\right)\\
&=\eta'_s\sin((k-1)\theta+\phi_s')
\end{align*}
where $\eta_s':=\sqrt{\widetilde{\alpha}^2_s+\widetilde{\beta}^2_s}$ and $\phi_s'$ are real functions of $\lambda$.
Then, using Lemma~\ref{per36}, we obtain
\begin{align*}
u_n&=Im(\varphi_n)\omega_n\\
&=\left(\eta'_s\sin((k-1)\theta+\phi'_s)\right)\left(\frac{2C(\lambda;T)}{(\alpha-1)T
n^{\alpha-1}}+O\left(\frac{1}{n^\alpha}\right)\right)\\
&=\frac{\widetilde{\eta}_s\sin((k-1)\theta+\phi_s')}{n^{\alpha-1}}+O\left(\frac{1}{n^\alpha}\right),
\end{align*}
where $\widetilde{\eta}_s:=\frac{C(\lambda;T)\eta_s'}{2(\alpha-1)T}$.
Finally, we wish to express our eigenvector in terms of $n$. Thus,
\begin{align*}
u_n&=\frac{\widetilde{\eta}_s\sin((k-1)\theta+\phi_s')}{n^{\alpha-1}}+O\left(\frac{1}{n^\alpha}\right)\\
&=\frac{\widetilde{\eta}_s\sin(n\theta/T+\widetilde{\zeta}_s)}{n^{\alpha-1}}+O\left(\frac{1}{n^\alpha}\right),
\end{align*}
where $\widetilde{\zeta}_s:=\phi_s'-s\theta/T.$
\end{proof}

\section{The structure of the potential and its asymptotics}

The following theorem gives an explicit formula for the potential, and the eigenvector, in terms of the solutions $\varphi_n$ of the periodic
problem, $\lambda$ and the parameter $\alpha$.

\begin{thm}\label{per39} Let $\lambda\in\sigma_{ell} (J_T)$ with $C(\lambda;T)\neq 0$. Define $\omega_n(\lambda)$ as in \eqref{4.5} and
$\varphi_n(\lambda)$ as in \eqref{4.6} and let $\alpha>\frac{3}{2}, n=T(k-1)+s,s\in\{0,\dots,T-1\}$, \begin{equation}\label{4.2}
q_n=-a_{n-1}(Im(\varphi_{n-1}))^2\left(\frac{n^{-\alpha}}{\omega_n}\right)+a_{n}(Im(\varphi_{n+1}))^2\left(\frac{(n+1)^{-\alpha}}{\omega_n}\right).\end{equation}
Then \begin{equation}\label{4.1} u_n(\lambda)=\omega_n(\lambda)Im\left(\varphi_n(\lambda)\right)\end{equation} satisfies
\begin{equation}\label{4.3} a_{n-1}u_{n-1}+a_{n}u_{n+1}+(q_n-\lambda) u_n=0 \end{equation} for $n\geq 2$. Moreover, $q_n$ has the following
asymptotic behaviour:\begin{equation} \label{4.4}
q_n=\frac{1}{n}\left(\rho_s(\lambda)\sin\left(2n\theta(\lambda)/T+\zeta_s(\lambda)\right)+\delta_s(\lambda)\right)+O\left(\frac{1}{n^2}\right),\end{equation}
where $\rho_s,\zeta_s$ and $\delta_s$ are real functions.
\end{thm}

\begin{remark}
In Formula~\eqref{4.2} we assume without loss of generality that $\omega_n\neq 0~\forall n=1,2,\dots.$ Indeed, due to the condition that
$C(\lambda;T)\neq 0$ and Formula~\eqref{asyp1} we see that $\omega_n\neq 0~\forall n\geq L$, where $L$ is sufficiently large. If $\omega_n$
vanishes for some $n<L$, then one can change the ansatz for $\omega_n$,~\eqref{4.5}, by introducing into the sum over $m$ an extra multiple,
$c_m$, where $c_m=1~\forall~m\geq L$. The values $c_1,c_2,\dots, c_{L-1}$ can be chosen in a suitable way such that
$\omega_1,\omega_2,\dots,\omega_{L-1}$ are not equal to zero.
\end{remark}

\begin{remark}
It can be shown by a lengthy calculation that if $C(\lambda;T)\neq 0$ then $$\sum\limits_{s=0}^{T-1} |\rho_s(\lambda)|^2>0,$$ showing that
the potential is genuinely of the form $\frac{1}{n}$ times an oscillating term. This also follows, without any calculation, from the fact
that no eigenvalues can be embedded in $\sigma_{ell}(J_T)$ by a potential $(q_n)$ with $q_n=O(\frac{1}{n^2})$. The proof of this fact for $T=1$ (the discrete Schr\"{o}dinger case) is well-known, see, e.g.~\cite{9g}.
% and the only exceptional value of $\lambda$ is zero.
\end{remark}

\begin{proof}
{\underline{(Step One)}} Check $u_n(\lambda)$ in \eqref{4.1} satisfies \eqref{4.3}. Then, for $n\geq 2$,
\begin{align}
&~~~~~~~~a_{n-1}u_{n-1}+a_{n}u_{n+1}-\lambda u_n=-q_nu_n\nonumber\\
&\iff
a_{n-1}\omega_{n-1}Im\left(\varphi_{n-1}\right)+a_{n}\omega_{n+1}Im\left(\varphi_{n+1}\right)-\lambda\omega_nIm\left(\varphi_n\right)\nonumber\\
&\hspace{30pt}=-q_n\omega_nIm\left(\varphi_n\right)\nonumber\\
&\iff Im(a_{n-1}\varphi_{n-1}+a_{n}\varphi_{n+1}-\lambda
\varphi_n)\omega_n+Im\left(\varphi_{n-1}\right)a_{n-1}\left(\omega_{n-1}-\omega_n\right)\nonumber\\
&~~~~~~~~~~~~~~~~~~~~~~+a_{n}Im\left(\varphi_{n+1}\right)\left(\omega_{n+1}-\omega_n\right)=-q_nIm\left(\varphi_n\right)\omega_n\nonumber\\
&\iff Im\left(\varphi_{n-1}\right)a_{n-1}\left(\omega_{n-1}-\omega_n\right)\nonumber\\
&~~~~~~~~~~~~~~~~~~~~~~+a_{n}Im\left(\varphi_{n+1}\right)\left(\omega_{n+1}-\omega_n\right)=-q_nIm\left(\varphi_n\right)\omega_n\label{journal1}
\end{align}
where we have used that \begin{equation} \omega_{n-1}-\omega_n=n^{-\alpha}Im\left(\varphi_n\right)Im\left(\varphi_{n-1}\right),\end{equation}
and that $\varphi_n$ satisfies the three-term recurrence relation~\eqref{4.3}. Choosing $q_n$ as in \eqref{4.2} guarantees the equality
\eqref{journal1}.

 Since $\lambda\in\sigma_{ell}(J_T)$, $|\varphi_n|=|{\varphi}_{s}|$, which means that $a_{n-1}\left(Im(\varphi_{n-1})\right)^2$ and
 $a_n\left(Im(\varphi_{n+1})\right)^2$ are oscillating factors in the variable $k$. Then one can expect growth or decay in $q_n$ to come from
 the components $\left(\frac{n^{-\alpha}}{\omega_n}\right)$ and $\left(\frac{(n+1)^{-\alpha}}{\omega_n}\right)$.
 By Lemma~\ref{per36} we have the relation $\omega_n\asymp n^{1-\alpha}$, and so we obtain
$$\frac{(n+1)^{-\alpha}}{\omega_n}\asymp n^{-1},$$ which gives a Coulomb-type decay for $q_n$.

{\underline{(Step Two)}}  We now prove \eqref{4.4}. Using Lemmas~\ref{lem39} and \ref{per36} for $n=T(k-1)+s, s\in\{1,\dots,T-1\}$, we
immediately obtain
\begin{align}
q_n&=-a_{s-1}\left(\eta_{s-1}\sin(2(k-1)\theta+\phi_{s-1})+\gamma_{s-1}\right)\left(\frac{2(\alpha-1)T}{nC(\lambda;T)}+O\left(\frac{1}{n^2}\right)\right)\nonumber\\
&+a_s\left(\eta_{s+1}\sin(2(k-1)\theta+\phi_{s+1})+\gamma_{s+1}\right)\left(\frac{2(\alpha-1)T}{nC(\lambda;T)}+O\left(\frac{1}{n^2}\right)\right)\label{4.15}\\
&=\frac{2(\alpha-1)T}{nC(\lambda;T)}\bigg(\left(-\eta_{s-1}a_{s-1}\cos\phi_{s-1}+\eta_{s+1}a_s\cos\phi_{s+1}\right)\sin(2(k-1)\theta)\nonumber\\
&\hspace{20pt}+\left(\eta_{s+1}a_s\sin\phi_{s+1}-\eta_{s-1}a_{s-1}\sin\phi_{s-1}\right)\cos(2(k-1)\theta)\nonumber\\
&\hspace{20pt}-\gamma_{s-1}a_{s-1}+\gamma_{s+1}a_s\bigg)+O\left(\frac{1}{n^2}\right)\nonumber\\
&=\frac{1}{n}\left(\left[\rho_s(\lambda)\sin(2(k-1)\theta(\lambda)+\zeta_s'(\lambda))\right]+\delta_s\right)+O\left(\frac{1}{n^2}\right),\nonumber
\end{align}
for real functions $\zeta_s'$ of $\lambda$,

\begin{multline*}
\rho_s^2:=\frac{4T^2({\alpha-1})^2}{C(\lambda;T)^2}\bigg(\left(\eta_{s+1}a_s\cos\phi_{s+1}-\eta_{s-1}a_{s-1}\cos\phi_{s-1}\right)^2\\
+\left(\eta_{s+1}a_s\sin\phi_{s+1}-\eta_{s-1}a_{s-1}\sin\phi_{s-1}\right)^2\bigg),
\end{multline*} and $\delta_s:=\frac{2T({\alpha-1})}{C(\lambda;T)}\left(-\gamma_{s-1}a_{s-1}+\gamma_{s+1}a_s\right)$.

For the special cases of $s\in\{0,T-1\}$ we must be careful because $n-1$ and $n+1$ will produce different values in the parameter $k$ to
those contained in the $n$-th element.
    When $s=0$ (i.e. $n=(k-1)T$):
    \begin{align*}
    q_n&=-a_{T-1}\left(\eta_{T-1}\sin(2(k-2)\theta+\phi_{T-1})+\gamma_{T-1}\right)\left(\frac{2(\alpha-1)T}{nC(\lambda;T)}+O\left(\frac{1}{n^2}\right)\right)\\
&~~~~~~~~~~~~+a_T\left(\eta_{1}\sin(2(k-1)\theta+\phi_{1})+\gamma_{1}\right)\left(\frac{2(\alpha-1)T}{nC(\lambda;T)}+O\left(\frac{1}{n^2}\right)\right)\\
    &=O\left(\frac{1}{n^2}\right)+\frac{2(\alpha-1)T}{nC(\lambda;T)}\bigg(-\eta_{T-1}a_{T-1}\sin(2(k-1)\theta+\phi_{T-1}-2\theta)\\
    &~~~~~~~+\eta_{1}a_0\sin(2(k-1)\theta+\phi_1)-a_{T-1}\gamma_{T-1}+a_{T}\gamma_1\bigg).
     \end{align*} This is of the same form as \eqref{4.15}. Consequently,
     $$q_n=O\left(\frac{1}{n^2}\right)+\frac{1}{n}\left[\rho_0(\lambda)\sin\left(2(k-1)\theta(\lambda)+\zeta_0'\right)+\delta_0(\lambda)\right],$$
     for functions $\zeta_0',\delta_0:=\frac{2T({\alpha-1})}{C(\lambda;T)}\left(-\gamma_{T-1}a_{T-1}+\gamma_{1}a_T\right)$ and

     $$\rho^2_0:=\frac{4T^2({\alpha-1})^2}{C(\lambda;T)^2}\big(\left(\eta_{1}a_T\cos\phi_{1}-\eta_{T-1}a_{T-1}\cos(\phi_{T-1}-2\theta)\right)^2$$
     $$~~~~~~~~~~~~~~~~~~~~~~~~~~~~~~~+\left(\eta_{1}a_T\sin\phi_{1}-\eta_{T-1}a_{T-1}\sin(\phi_{T-1}-2\theta)\right)^2\big).$$
Similarly, when $s=T-1$ (i.e. $n=kT-1$):
\begin{align*}
 q_n&=-a_{T-2}\left(\eta_{T-2}\sin(2(k-1)\theta+\phi_{T-2})+\gamma_{T-2}\right)\left(\frac{2(\alpha-1)T}{nC(\lambda;T)}+O\left(\frac{1}{n^2}\right)\right)\\
&~~~~~~~~~~~~~~~~~~~~~+a_T\left(\eta_{0}\sin(2k\theta+\phi_{0})+\gamma_{0}\right)\left(\frac{2(\alpha-1)T}{nC(\lambda;T)}+O\left(\frac{1}{n^2}\right)\right)\\
&=O\left(\frac{1}{n^2}\right)+\frac{1}{n}\left(\rho_{T-1}(\lambda)\sin\left(2(k-1)\theta(\lambda)+\zeta_{T-1}'\right)+\delta_{T-1}(\lambda)\right),
 \end{align*}
 for functions
 $\zeta_{T-1}',\delta_{T-1}:=\frac{2T({\alpha-1})}{C(\lambda;T)}\left(-\gamma_{T-2}a_{T-2}+\gamma_{0}a_{T-1}\right)$ and
  $$\rho_{T-1}:=\frac{4T^2({\alpha-1})^2}{C(\lambda;T)^2}\big(\left(\eta_{0}a_{T-1}\cos(\phi_{0}+2\theta-\eta_{T-2}a_{T-2}\cos\phi_{T-2})\right)^2$$
  $$~~~~~~~~~~~~~~~~~~~~~~~~~~~~~~~~+\left(\eta_{0}a_{T-1}\sin(\phi_{0}+2\theta)-\eta_{T-2}a_{T-2}\sin\phi_{T-2}\right)^2\big).$$
 Thus, for all $s\in\{0,\dots,T-1\}$, we have the result:
 $$q_n=O\left(\frac{1}{n^2}\right)+\frac{1}{n}\left[\rho_s(\lambda)\sin\left(2(k-1)\theta(\lambda)+\zeta_s'\right)+\delta_s(\lambda)\right].$$
However, we still wish to express our potential in terms of the variable $n$. This follows simply from defining the new function $\zeta_s$,
where $\zeta_s:=\zeta_s'-s\theta/T.$
\end{proof}

\begin{remark}
Concerning the roots of $C(\lambda;T)$ for $\lambda\in\sigma_{ell}(J_T)$, we may say the following. As has been stated in the theorem,
using a sufficiently slowly decaying potential, $q_n=O\left(\frac{1}{n}\right)$, it is possible to introduce a subordinate $l^2$-solution for
any fixed $\lambda\in\sigma_{ell}(J_T)$, except at roots of $C(\lambda;T)$. However we believe that at any root of $C(\lambda;T)$,
$\lambda\in\sigma_{ell}(J_T)$, the existence of the subordinate $l^2$-solution can still be obtained by using a potential,
$q_n=O\left(\frac{1}{n}\right)$, $n\rightarrow\infty$.
\end{remark}

\begin{remark}
The last statement in the previous remark is true for the case of the discrete Schr\"{o}dinger operator. To see this choose a candidate eigenvector of the form
$u_n=\frac{(-1)^{\left\lfloor{\frac{n}{2}}\right\rfloor}}{n}$ for $n\geq 1$ and a potential defined by $q_n=\frac{2n(-1)^n}{n^2-1},$ for $n\geq2$ and $q_1=\frac{1}{2}$. Then clearly $u_2+q_1u_1=0$ and the recurrence equations $$\frac{(-1)^{\lfloor{\frac{n+1}{2}}\rfloor}}{n+1}+\frac{(-1)^{\lfloor{\frac{n-1}{2}}\rfloor}}{n-1}+\frac{2n(-1)^n}{n^2-1}\frac{(-1)^{\lfloor{\frac{n}{2}}\rfloor}}{n}=0,$$ are also satisfied for $n\geq 2$. Thus the eigenvalue $\lambda=0$ becomes embedded in the a.c.~spectrum of the operator.
\end{remark}

\section{Embedded eigenvalues}
Theorem~\ref{per39} guarantees a subordinate solution of the recurrence relation~\eqref{4.3}, but does not guarantee an embedded eigenvalue
since it still remains to be seen if the first-row equation of the Jacobi matrix is satisfied, i.e. $$q_1u_1+a_1u_2=\lambda u_1.$$

The next result shows that it is always possible to make $\lambda$ an eigenvalue by suitably modifying the potential, slightly.

\begin{thm}
Assume $\lambda\in \sigma_{ell}(J_T), C(\lambda;T)\neq 0, \alpha>\frac{3}{2}$ and let $u_n$ be given by \eqref{4.1} for $n\geq 2$ and $q_n$
by~\eqref{4.2} for $n\geq 3$. Then it is possible to choose $u_1,q_1,q_2\in\mathbb{R}$ such that $\lambda\in\sigma_p(J_T+Q),$ where $Q$ is an
infinite diagonal matrix with entries $(q_n)$.
 \end{thm}

\begin{remark}
Note that $$\sigma_{ess}(J_T)=\sigma_{a.c.}(J_T)=\sigma_{a.c.}(J_T+Q)=\sigma_{ess}(J_T+Q)$$
and $$\sigma_{ell}(J_T)=\sigma_{ell}(J_T+Q).$$
The coinciding of the essential spectrum of $J_T$ and $J_T+Q$ follows from the classical Weyl Theorem~\cite{9l}. The preservation of the a.c.
spectrum under the perturbation $q_n$ follows from the combination of subordinancy theory~\cite{7} and asymptotic Levinson-type
theory~\cite{9m}. A similar result for the continuous Schr\"{o}dinger case was proven by Behncke~\cite{9n}.
\end{remark}
\begin{proof}

By Theorem~\ref{per39} we have $$a_{n-1}u_{n-1}+a_nu_{n+1}+(q_n-\lambda)u_n=0$$ for $n\geq 3$. However, we also need to satisfy

\begin{equation}\label{exact1}
q_1u_1+a_1u_2=\lambda u_1
\end{equation}
and
\begin{equation}
a_1u_1+(q_2-\lambda)u_2+a_2u_3=0.
\end{equation}
We have two cases:
\begin{enumerate}
\item
If $u_2\neq 0$ then defining $q_2:=\frac{-\lambda u_2-a_2u_3-a_1u_1}{u_2}$ with $u_1:=-\frac{a_1u_2}{q_1-\lambda}$, with $q_1$ as a free
parameter and not equal to $\lambda$, ensures all conditions are satisfied.
\item
If $u_2=0$ then defining $u_1:=-\frac{a_2u_3}{a_1}$ and $q_1:=\lambda$, with $q_2$ as a free parameter, ensures all conditions are
satisfied.\qedhere
\end{enumerate}
\end{proof}

We are grateful to the unknown referee for his/her very useful remarks.

\end{document}